\documentclass[12pt, reqno]{amsart}
\usepackage{amsmath, amsthm, amscd, amsfonts, amssymb, graphicx, color}
\usepackage[bookmarksnumbered, colorlinks, plainpages]{hyperref}
\hypersetup{colorlinks=true,linkcolor=red, anchorcolor=green, citecolor=cyan, urlcolor=red, filecolor=magenta, pdftoolbar=true}
\input{mathrsfs.sty}

\newtheorem{theorem}{Theorem}[section]
\newtheorem{lemma}[theorem]{Lemma}
\newtheorem{proposition}[theorem]{Proposition}
\newtheorem{corollary}[theorem]{Corollary}
\theoremstyle{definition}
\newtheorem{definition}[theorem]{Definition}
\newtheorem{example}[theorem]{Example}

\theoremstyle{remark}

\numberwithin{equation}{section}
\makeatletter
\newcommand*{\rom}[1]{\expandafter\@slowromancap\romannumeral #1@}
\makeatother
\begin{document}

\title[Extensions of the Hilbert-multi-norm in Hilbert $C^*$-modules]{Extensions of the Hilbert-multi-norm in Hilbert $C^*$-modules}

\author[S. Abedi, M.S. Moslehian]{Sajjad Abedi$^1$ \MakeLowercase{and} Mohammad Sal Moslehian$^2$}

\address{$^1$Department of Pure Mathematics, Ferdowsi University of Mashhad, P. O. Box 1159, Mashhad 91775, Iran}
\email{sajjad.abedi@mail.um.ac.ir; sajjadabedi2022@gmail.com}

\address{$^2$Department of Pure Mathematics, Center of Excellence in Analysis on Algebraic Structures (CEAAS), Ferdowsi University of Mashhad, P. O. Box 1159, Mashhad 91775, Iran.}
\email{moslehian@um.ac.ir; moslehian@yahoo.com}

\renewcommand{\subjclassname}{\textup{2020} Mathematics Subject Classification}


\begin{abstract}
Dales and Polyakov introduced a multi-norm $\left( \left\|\cdot\right\|_n^{(2,2)}:n\in\mathbb{N}\right)$ based on a Banach space $\mathscr{X}$ and showed that it is equal with the Hilbert-multi-norm $\left( \left\|\cdot\right\|_n^{\mathscr{H}}:n\in\mathbb{N}\right)$ based on an infinite-dimensional Hilbert space $\mathscr{H}$. We enrich the theory and present three extensions of the Hilbert-multi-norm for a Hilbert $C^*$-module $\mathscr{X}$. We denote these multi-norms by $\left( \left\|\cdot\right\|_n^{\mathscr{X}}:n\in\mathbb{N}\right)$, $\left( \left\|\cdot\right\|_n^{*}:n\in\mathbb{N}\right)$, and $\left( \left\|\cdot\right\|_n^{\mathcal{P}\left(\mathfrak{A} \right) }:n\in\mathbb{N}\right)$. We show that $\left\|x\right\|_n^{\mathcal{P}\left(\mathfrak{A} \right) }\geq\left\|x\right\|_n^{\mathscr{X}}\leq \left\|x\right\|_n^{*}$ for each $x\in\mathscr{X}^n$. In the case when $\mathscr{X}$ is a Hilbert $\mathbb{K}\left(\mathscr{H}\right)$-module, for each $x\in\mathscr{X}^n$, we observe that $\left\|\cdot\right\|_n^{\mathcal{P}\left(\mathfrak{A} \right)}=\left\|\cdot\right\|_n^{\mathscr{X}}$. Furthermore, if $\mathscr{H}$ is separable and $\mathscr{X}$ is infinite-dimensional, we prove that $\left\|x\right\|_n^{\mathscr{X}}=\left\|x\right\|_n^{*}$. Among other things, we show that small and orthogonal decompositions with respect to these multi-norms are equivalent. Several examples are given to support the new findings.
\end{abstract}
\maketitle
\section{Introduction and preliminaries}

Suppose that $\mathfrak{A}$ is a $C^*$-algebra. Let $\mathscr{X}$ be a (right) \emph{Hilbert $C^*$-module} over $\mathfrak{A}$. A closed submodule $\mathscr{Y}$ of $\mathscr{X}$ is said to be \emph{orthogonally complemented} if $\mathscr{X}=\mathscr{Y}\oplus \mathscr{Y}^{\perp}$, where $\mathscr{Y}^{\perp}=\left\lbrace y\in\mathscr{X}:\left\langle x,y\right\rangle=0, x\in\mathscr{Y}\right\rbrace$. Let $\mathscr{Y}$ and $\mathscr{Z}$ be closed submodules of $\mathscr{X}$. If $\mathscr{Z}\subseteq\mathscr{Y}^{\perp}$, then we denote it by $\mathscr{Y}\perp\mathscr{Z}$. For $\left| x \right|= \left\langle x, x \right\rangle^{1/2}$, the $\mathfrak{A}$-valued triangle inequality $\left|x+y\right|\leq\left|x \right|+\left|y\right|$ may not be valid, but the Cauchy--Schwarz inequality $\left|\left\langle x, y \right\rangle \right|^2\leq\left\| x \right\|^2\left|y \right|^2$ holds; see \cite[Chapter 1]{lance}. A $C^*$-algebra $\mathfrak{A}$ is a Hilbert $\mathfrak{A}$-module with the inner product $\left\langle a,b\right\rangle=a^*b$ for each $a,b\in\mathfrak{A}$. The direct sum $\mathit{l}_n^2(\mathscr{X})$ of $n$-copies of a Hilbert $\mathfrak{A}$-module $\mathscr{X}$ is a Hilbert $\mathfrak{A}$-module under the inner product given by 
\begin{align*}
	\left\langle(x_1,\dots,x_n),(y_1,\dots,y_n)\right\rangle_{\mathit{l}_n^2(\mathscr{X})}=\sum_{i=1}^{n} \left\langle x_i,y_i\right\rangle\,.
\end{align*} 
 We shall abbreviate $\mathit{l}_n^2(\mathbb{C})$ by $\mathit{l}_n^2$. A map $T : \mathscr{X}\rightarrow \mathscr{Y}$ is said to be \emph{adjointable} if there exists a map $T^* : \mathscr{Y}\rightarrow \mathscr{X}$ such that $\left\langle Tx,y \right\rangle=\left\langle x,T^*y \right\rangle$ for all $x \in \mathscr{X}$ and $y \in \mathscr{Y}$. The set of all adjointable maps from $\mathscr{X}$ into $\mathscr{Y}$ is denoted by $\mathcal{L}(\mathscr{X},\mathscr{Y})$. Let $x \in \mathscr{X}$ and let $y \in \mathscr{Y}$. Set ${\theta}_{y,x}:\mathscr{X} \rightarrow \mathscr{Y}$ by ${\theta}_{y,x}(z)\mapsto y\left\langle x,z\right\rangle \quad (z \in \mathscr{X})$. It is easy to verify that ${\theta}_{y,x} \in \mathcal{L}(\mathscr{X},\mathscr{Y})$ and $\left( {\theta}_{y,x}\right)^*={\theta}_{x,y}$. The norm-closed linear span of $\left\lbrace {\theta}_{y,x}: x \in \mathscr{X}, y \in \mathscr{Y} \right\rbrace $ is denoted by $\mathcal{K}(\mathscr{X},\mathscr{Y})$, the space of ``compact'' operators. We shall abbreviate $\mathcal{L}(\mathscr{X},\mathscr{X})$ and $\mathcal{K}(\mathscr{X},\mathscr{X})$ by $\mathcal{L}(\mathscr{X})$, $\mathcal{K}(\mathscr{X})$, respectively. It is known that the space $\mathcal{L}(\mathscr{X})$ is a $C^*$-algebra. The identity map in $\mathcal{L}(\mathscr{X})$ is denoted by $I_{\mathscr{X}}$. We refer the readers to \cite{FRA, lance, Manu} for more information about Hilbert $C^*$-modules.

 Dales and Polyakov \cite{2012} introduced the notion of multi-norm based on a normed space
$\mathcal{X}$. A \emph{multi-norm} based on $\mathcal{X}$ is a sequence $(\left\|\cdot\right\|_n : n\in \mathbb{N})$ such that $\left\|\cdot\right\|_n$ is a norm on	${\mathcal{X}}^n$, for each $x_0 \in \mathcal{X}$,	$\left\|x_0\right\|_1=\left\|x_0\right\|$, 	and the following Axioms hold for every $n \in \mathbb{N}$ and $x=\left( x_1,\dots,x_n\right) \in {\mathcal{X}}^n$:
\begin{enumerate}
\item[(A1)]
For each permutation $\sigma$ on $\left\lbrace 1,\dots,n\right\rbrace $, it holds that $\left\|\left( x_{\sigma(1)},\dots,x_{\sigma(n)}\right) \right\|_n = \left\| x \right\|_n$;
\item[(A2)]
$\left\|\left( \alpha_1x_1,\dots,\alpha_nx_n\right)\right\|_n\leq (\max_{1\leq i\leq n } \left| \alpha_i \right| )\left\| x \right\|_n$ for all $\alpha_1,\dots,\alpha_n \in \mathbb{C}$;
\item[(A3)]
 $\left\|\left( x_1,\dots,x_{n-1},0\right) \right\|_n=\left\|\left( x_1,\dots,x_{n-1}\right) \right\|_{n-1}$;
\item[(A4)]
$\left\|\left( x_1,\dots,x_{n-1},x_{n-1}\right)\right\|_n\leq\left\|\left( x_1,\dots,x_{n-2},x_{n-1}\right) \right\|_{n-1}$.
\end{enumerate}
In this case, $\left( \left\|\cdot\right\|_n: n\in\mathbb{N} \right)$ is called a \emph{multi-norm} based on $\mathcal{X}$, and $\left( \left( {\mathcal{X}}^n,\left\|\cdot\right\|_n\right):n\in\mathbb{N} \right) $ is a \emph{multi-normed space}. The \emph{minimum multi-norm} based on a normed space $\mathcal{X}$ is the sequence $\left(\left\|\cdot\right\|_n:n\in\mathbb{N}\right)$ defined by
$\left\|(x_1,\ldots,x_{n})\right\|_{n}=\max_{1\leq i\leq n}\left\|x_i\right\|$. Dales and Polyakov introduced \emph{the weak $2$-summing norm} based on a normed space $\mathcal{X}$ \cite[Definition 3.15]{2012} such that for each $n \in \mathbb{N}$ and $x=(x_1,\dots, x_n)\in \mathcal{X}^n$,
\begin{equation*}
\mu_{2,n}(x_1,\dots,x_n)=\sup\left\lbrace \left( \sum_{i=1}^{n}| f(x_i)|^2\right)^{1/2}:f\in\mathcal{X}^{\prime}\right\rbrace,
\end{equation*}
where $\mathcal{X}^{\prime}$ is the Banach dual space of $\mathcal{X}$. Furthermore, they defined \emph{the $(2,2)$-multi-norm} based on $\mathcal{X}$ by
\begin{equation}\label{(2,2)}
	\left\|x\right\|^{(2,2)}_n=\sup\left\{\left(\sum_{i=1}^n\left|f_i(x_i) \right|^{2}\right)^{1/2}: 
	\mu_{2,n}(f_1,\dots,f_n)\leq 1\right\},
\end{equation}
in which the supremum is taken over $f_1,\dots,f_n \in \mathcal{X}^{\prime}$. It deduces from \cite[Theorem 4.1]{2012} that $\left( \left(\mathcal{X}^n,\left\|\cdot\right\|_n^{(2,2)}\right):n\in\mathbb{N}\right)$ is a multi-normed space. We refer the readers to \cite{2012, blasco} for more information about multi-norms.

An extension of $\mu_{2,n}(x_1,\dots,x_n)$ in the setting of Hilbert $C^*$-modules is presented in \cite[Proposition 3.4(1)]{sajjad} as follows:
\begin{align}\label{mos4}
\mu^*_{n}(x_1,\dots,x_n)&=\sup\left\lbrace \left( \sum_{i=1}^{n} |T(x_i)^*|^2\right)^{1/2} : T \in{\mathcal{L}(\mathscr{X}, \mathfrak{A})}_{[1]}\right\rbrace\nonumber\\&
=\sup_{y\in\mathscr{X}_{[1]}}\left\|\left( \sum_{i=1}^{n}\left|\left\langle x_i,y\right\rangle\right|^2\right)^{1/2}\right\|,
\end{align}
where $\mathscr{X}_{[1]}$ is the norm-closed unit ball of $\mathscr{X}$. Moreover, $\mu^*_{n}$ enjoys the following properties:
	 \begin{equation}\label{powernorm}
\max_{1\leq i\leq n}\left\|x_i\right\|\leq\mu^*_{n}(x_1,\dots,x_n)\leq\sum_{i=1}^{n}\left\|x_i\right\|;
	 \end{equation}
	 \begin{equation}\label{Proposition 3.4(1)}
	\mu^*_{n}(x_1a_1,\dots,x_na_n)\leq\max_{1\leq i\leq n}\left\|a_i\right\|\mu^*_{n}(x_1,\dots,x_n)\quad\left( a_1,\dots,a_n\in\mathfrak{A}\right) ;
	\end{equation}
	 \begin{equation}\label{Lemma3.7}
		\mu^*_{n}(x_1,\dots,x_n)=\sup\left\lbrace\left\| \sum_{i=1}^{n}x_ia_i\right\| : a_1,\dots,a_n\in\mathfrak{A},\left\|\sum_{i=1}^{n}a_i^*a_i \right\|\leq 1\right\rbrace;
	\end{equation}
	 \begin{equation}\label{Proposition 3.6}
			\mu^*_{n}(x_1,\dots,x_n)=\min\left\lbrace\lambda>0:\left( \sum_{i=1}^{n}\left| \left\langle y_i,x\right\rangle\right|^2\right)^{1/2}\leq\lambda\left|x\right|\,\ \text{for all} x\in\mathscr{X}\right\rbrace.
	\end{equation}

In the next section, we introduce a new version of $\left( \left\|\cdot\right\|^{(2,2)}:n\in\mathbb{N}\right) $ based on a Hilbert $C^*$-module $\mathscr{X}$ as follows:
\begin{equation*}
\left\|\left( x_1,\dots,x_n\right) \right\|_n^{*}=\sup\left\{\left\| \sum_{i=1}^n \left|\left\langle x_i,y_i\right\rangle\right|^2\right\|^{1/2}: 	\mu^*_{n}(y_1,\dots,y_n)\leq 1\right\}.
\end{equation*}
We show that $\left( \left(\mathscr{X}^n,\left\|\cdot\right\|_n^{*}\right):n\in\mathbb{N}\right)$ is a multi-normed space based on $\mathscr{X}$.

In \cite[Theorem 4.15]{2012}, Dales and Polyakov presented the notion of \emph{Hilbert multi-norm} based on Hilbert spaces. We extend the Hilbert multi-norm to that of \emph{Hilbert $C^*$-multi-norm} based on $\mathscr{X}$ as follows:
\begin{equation}\label{Hilbert multi norm}
\left\|\left(x_1,\dots,x_n\right) \right\|^{\mathscr{X}}_n=\sup\left\|P_1x_1 +\dots+P_nx_n\right\|\quad\left( x_1,\dots,x_n\in\mathscr{X}\right),
\end{equation}
 where the supremum is taken over all families $(P_i:1\leq i\leq n)$ of mutually orthogonal projections in $\mathcal{L}(\mathscr{X})$ summing to $I_{\mathscr{X}}$ by allowing the possibility that $P_j=0$ for some $1 \leq j \leq n $. It immediately follows that $\left(\left\|\cdot\right\|^{\mathscr{X}}_n:n\in\mathbb{N}\right)$ is a multi-norm based on $\mathscr{X}$. We call it the \emph{Hilbert $C^*$-multi-norm} based on $\mathscr{X}$. Assume that $\mathscr{H}$ is a Hilbert space with o-dim$\left(\mathscr{H}\right)\geq n$, where o-dim$(\mathscr{H})$ stands for the Hilbert dimension; that is, the cardinal number of any orthonormal basis of $\mathscr{H}$. It is known from \cite[Theorem 4.19]{2012} that 
 \begin{equation}\label{dim}
\left\|x\right\|_n^{\mathscr{H}}= \left\|x\right\|_n^{*}\quad \left( x\in \mathscr{H}^n\right).
 \end{equation} 
 Equality \eqref{dim} is true for an arbitrary Hilbert space $\mathscr{H}$ in the case where $n=1, 2, 3$ \cite[p. 44]{hilbert}. However, it deduces from  \cite[Theorem 4.11]{hilbert} and \cite[Theorem 4.19]{2012} that the equality is not true in general  in the case when $n\geq 4$ and o-dim$\left(\mathscr{H}\right)\geq 3$.

 We write $\mathcal{P}\left(\mathfrak{A}\right) $ for the set of pure states on $\mathfrak{A}$. If $\mathfrak{A}$ is commutative, then $\mathcal{P}\left(\mathfrak{A}\right)$ is indeed the character space of $\mathfrak{A}$ equipped with the weak$\rm{^*}$ topology. 

Let $\tau\in\mathcal{P}\left(\mathfrak{A}\right)$. Set $\mathcal{N}_{\tau}=\left\lbrace x\in\mathscr{X}:\tau\left( \left\langle x,x\right\rangle\right) =0 \right\rbrace $. We consider an inner product $\langle\cdot,\cdot\rangle:\mathscr{X}/\mathcal{N}_{\tau}\times\mathscr{X}/\mathcal{N}_{\tau}\rightarrow\mathbb{C}$ such that $\left\langle x+\mathcal{N}_{\tau},y+\mathcal{N}_{\tau}\right\rangle=\tau\left( \left\langle x,y\right\rangle\right)$. The Hilbert completion of ${\mathscr{X}}/{\mathcal{N}_{\tau}}$ is denoted by $\mathscr{H}_{\tau}$. 

 For $x=\left( x_1,\dots,x_n\right) \in\mathscr{X}^n$, we define
\begin{equation*}
	\left\|x\right\|_n^{\mathcal{P}\left(\mathfrak{A} \right)}:=\sup_{\tau\in\mathcal{P}\left(\mathfrak{A}\right)}\left\|\left(x_1+\mathcal{N}_{\tau},\dots,x_n+\mathcal{N}_{\tau}\right)\right\|_n^{\mathscr{H}_{\tau}}.
\end{equation*}
We observe that $\left(	\left\|\cdot\right\|_n^{\mathcal{P}\left(\mathfrak{A} \right)}:n \in \mathbb{N} \right) $ is a multi-norm based on $\mathscr{X}$. In the next section, we see that $\left\|x\right\|_n^{\mathcal{P}\left(\mathfrak{A} \right) }\geq\left\|x\right\|_n^{\mathscr{X}}\leq \left\|x\right\|_n^{*}$ for each $x\in\mathscr{X}^n$. In the case when $\mathscr{X}$ is a Hilbert $\mathbb{K}\left(\mathscr{H}\right)$-module, we observe that $\left\|x\right\|_n^{\mathcal{P}\left(\mathfrak{A} \right)}=\left\|x\right\|_n^{\mathscr{X}}$ for each $x\in\mathscr{X}^n$. Moreover, if $\mathscr{H}$ is separable and	 $\text{o-dim}\left(\mathscr{X}\right)\geq n\left( \text{o-dim}\left(\mathscr{H}\right)\right)$, then we arrive at 	$\left\|x\right\|_n^{\mathscr{X}}=\left\|x\right\|_n^{*}$.

\section{Multi-norms based on a Hilbert $C^*$-module}
In this section, we extend the $(2,2)$-multi-norm \eqref{(2,2)} to set up of Hilbert $C^*$-modules. We start our work with the following definition.
\begin{definition}
Let $\mathscr{X}$ be a Hilbert $C^*$-module and let $x=(x_1,\dots,x_n)\in\mathscr{X}^n$. Then we set
\begin{equation}\label{mos3}
\left\|x\right\|_n^{*}:=\sup\left\{\left\| \sum_{i=1}^n\left|\left\langle y_i,x_i\right\rangle\right|^{2}\right\| ^{1/2}: 	\mu^*_{n}(y_1,\dots,y_n)\leq 1\right\}.
\end{equation}
\end{definition}
\begin{theorem}
	Suppose that $\mathscr{X}$ is a Hilbert $\mathfrak{A}$-module. Then $\left( \left(\mathscr{X}^n,\left\|\cdot\right\|_n^{*} \right):n\in\mathbb{N}\right) $ is a multi-normed space.
\end{theorem}
\begin{proof}
	Obviously, $\left(\mathscr{X}^n,\left\|\cdot\right\|_n^{*} \right) $
	satisfies Axioms (A1), (A2), and (A3). So it is sufficient to prove Axiom (A4). To this end, take $n\in \mathbb{N}$. Pick $x_1,\dots,x_n\in\mathscr{X}$ and $\varepsilon >0$. Then there are elements $y_1,\dots,y_{n+1}\in\mathscr{X}$ such that 
	$\mu^*_{n+1}(y_1,\dots,y_{n+1})\leq 1$ and
	\begin{align}\label{mos1}
		\left(\left\|\left( x_1,\dots,x_n,x_n\right)\right\|^{*}_{n+1}\right) ^2 -\frac{\varepsilon}{2}<\left\|\sum_{i=1}^{n-1}\left|\left\langle y_i,x_i\right\rangle\right| ^{2}+\left| \left\langle y_n,x_n\right\rangle\right| ^{2} 
		+\left| \left\langle y_{n+1},x_n\right\rangle\right| ^{2}\right\|.
	\end{align}
	Suppose that $\mathfrak{A}$ is unital. Set $a=\left(\left| \left\langle y_n,x_n\right\rangle\right| ^{2} 
	+\left| \left\langle y_{n+1},x_n\right\rangle\right| ^{2}\right)^{1/2}$. Also, set
	\begin{equation}
		y=(a+\varepsilon_0)^{-1}\left\langle x_n,y_n\right\rangle y_n+(a+\varepsilon_0)^{-1}\left\langle x_n,y_{n+1}\right\rangle y_{n+1},
	\end{equation}
	where $\varepsilon_0={\varepsilon}/\left( 4\left\|a \right\|+\varepsilon\right)$. Let $x\in\mathscr{X}$. It deduces from the Cauchy--Schwarz inequality that
	\begin{align*}
		\left|\left\langle y,x\right\rangle\right|^2&=
		\left| (a+\varepsilon_0)^{-1}\left\langle x_n,y_n\right\rangle \left\langle y_n,x\right\rangle+(a+\varepsilon_0)^{-1}\left\langle x_n,y_{n+1}\right\rangle\left\langle y_{n+1},x\right\rangle\right|^2\\&=\left| \left\langle \left(\left\langle y_n,x_n\right\rangle(a+\varepsilon_0)^{-1},\left\langle y_{n+1},x_n\right\rangle(a+\varepsilon_0)^{-1} \right),\left(\left\langle y_{n},x\right\rangle,\left\langle y_{n+1},x\right\rangle \right) \right\rangle_{\mathit{l}^2_2(\mathfrak{A})}\right|^2 \\&\leq\left\|\left(\left\langle y_n,x_n\right\rangle(a+\varepsilon_0)^{-1},\left\langle y_{n+1},x_n\right\rangle(a+\varepsilon_0)^{-1} \right) \right\|_{\mathit{l}^2_2(\mathfrak{A})}^2 \left|\left(\left\langle y_{n},x\right\rangle,\left\langle y_{n+1},x\right\rangle \right) \right|_{\mathit{l}^2_2(\mathfrak{A})}^2 \\&=\left\|(a+\varepsilon_0)^{-1}a^2 (a+\varepsilon_0)^{-1}\right\| \left(\left| \left\langle y_n,x\right\rangle\right| ^{2} 
		+\left| \left\langle y_{n+1},x\right\rangle\right| ^{2}\right)\\&\leq\left| \left\langle y_n,x\right\rangle\right| ^{2} 
		+\left| \left\langle y_{n+1},x\right\rangle\right| ^{2}.\tag{by using the functional calculus for $a$} 
	\end{align*}
	It deduces that $\mu^*_{n}(y_1,\dots,y_{n-1},y)\leq \mu^*_{n+1}(y_1,\dots,y_{n+1})\leq 1$. Furthermore, we observe that 
	\begin{equation}
		\left\langle y,x_n\right\rangle=
		(a+\varepsilon_0)^{-1}\left(\left| \left\langle y_n,x_n\right\rangle\right| ^{2} 
		+\left| \left\langle y_{n+1},x_n\right\rangle\right| ^{2}\right) =(a+\varepsilon_0)^{-1}a^2.
	\end{equation}
	Since $a\geq0$, we have $a-(a+\varepsilon_0)^{-1}a^2\leq \varepsilon_0$ and so
	\begin{equation}\label{mos2}
		a^2-\left((a+\varepsilon_0)^{-1}a^2\right)^2\leq \varepsilon_0\left(a+(a+\varepsilon_0)^{-1}a^2\right)\leq2\left\|a \right\|\varepsilon_0\leq\frac{\varepsilon}{2}.
	\end{equation}
	We therefore arrive at $	\left| \left\langle y,x_n\right\rangle\right|^2
	=\left((a+\varepsilon_0)^{-1}a^2\right)^2\geq a^2-{\varepsilon}/{2}$. Thus
	\begin{align*}
		\left(\left\|\left( x_1,\dots,x_n,x_n\right)\right\|^{*}_{n+1}\right) ^2 -\varepsilon&\leq\left\|\sum_{i=1}^{n-1}\left|\left\langle y_i,x_i\right\rangle\right| ^{2}+\left| \left\langle y_n,x_n\right\rangle\right| ^{2} 
		+\left| \left\langle y_{n+1},x_n\right\rangle\right| ^{2}\right\|-\frac{\varepsilon}{2}\\
		&\qquad\qquad\qquad\qquad\qquad\qquad \qquad\qquad ({\rm by \eqref{mos1} })\\
		&\leq\left\|\sum_{i=1}^{n-1}\left|\left\langle y_i,x_i\right\rangle\right| ^{2}+a^2\right\|-\frac{\varepsilon}{2}\\
		&\leq\left\|\sum_{i=1}^{n-1}\left|\left\langle y_i,x_i\right\rangle\right| ^{2}+ \left| \left\langle y,x_n\right\rangle\right|^{2}+\frac{\varepsilon}{2}\right\|-\frac{\varepsilon}{2}\\
		&\qquad\qquad\qquad\qquad\qquad\qquad \qquad\qquad ({\rm by \eqref{mos2} })\\
		&\leq\left\|\sum_{i=1}^{n-1}\left|\left\langle y_i,x_i\right\rangle\right| ^{2}+\left| \left\langle y,x_n\right\rangle\right|^{2}\right\|\leq \left\|\left( x_1,\dots,x_n\right)\right\|^{*2}_{n}.
	\end{align*}
	Thus, we arrive at (A4). In the case when $\mathfrak{A}$ is non-unital, we use the minimal unitization $\mathfrak{A}\oplus \mathbb{C}$ of $\mathfrak{A}$, and consider $\mathscr{X}$ as a Hilbert $\mathfrak{A}\oplus \mathbb{C}$-module via the right action $x\left(a,\alpha\right)=xa+\alpha x$, where $x\in\mathscr{X}$, $a\in\mathfrak{A}$, and $\alpha\in\mathbb{C}$. 
\end{proof}
\begin{proposition}
	Assume that $\mathscr{X}$ is a Hilbert $\mathfrak{A}$-module. Let $\mathscr{Y}$ be a closed submodule of $\mathscr{X}$. Pick $x = (x_1,\dots, x_n)\in F^n$. 
	Then $\left\|x\right\|_{n,\mathscr{Y}}^{*} \geq \left\|x\right\|_{n,\mathscr{X}}^{*}$. If $\mathscr{Y}$ is complemented in $\mathscr{X}$, then $\left\|x\right\|_{n,\mathscr{Y}}^{*}=\left\|x\right\|_{n,\mathscr{X}}^{*}$.
\end{proposition}
\begin{proof}
	 Let $y_1,\dots,y_n \in \mathscr{Y}$. It follows that $\mu_{n,\mathscr{Y}}^{*}(y_1\dots,y_n) \leq \mu_{n,\mathscr{X}}^{*}(y_1\dots,y_n)$.
	So $\left\|x\right\|_{n,\mathscr{Y}}^{*} \geq \left\|x\right\|_{n,\mathscr{X}}^{*}$.
	
	 Suppose that $\mathscr{Y}$ is complemented in $\mathscr{X}$. Set $P\in\mathcal{L}\left(\mathscr{X}\right)$ be the orthogonal projection onto $\mathscr{Y}$. Note that $P(\mathscr{X}_{[1]})=\mathscr{Y}_{[1]}$. We have
	\begin{align*}
		\mu^*_{n,\mathscr{X}}(y_1,\dots,y_n)&=\sup_{x\in\mathscr{X}_{[1]}}\left\|\left( \sum_{i=1}^{n}\left|\left\langle Py_i,x\right\rangle\right|^2\right)^{1/2}\right\|=\sup_{x\in\mathscr{X}_{[1]}}\left\|\left( \sum_{i=1}^{n}\left|\left\langle y_i,Px\right\rangle\right|^2\right)^{1/2}\right\|\\&=\sup_{y\in\mathscr{Y}_{[1]}}\left\|\left( \sum_{i=1}^{n}\left|\left\langle y_i,y\right\rangle\right|^2\right)^{1/2}\right\| \leq	\mu^*_{n,\mathscr{Y}}(y_1,\dots,y_n).
		\end{align*}
	Hence $\left\|x\right\|_{n,\mathscr{Y}}^{*}=\left\|x\right\|_{n,\mathscr{X}}^{*}$.
\end{proof}

Suppose that $T\in\mathcal{L}(\mathscr{X})$. Let $\phi\left(T\right)$ on $\mathscr{X}/\mathcal{N}_{\tau}$ be defined by $\phi(T)\left(x+\mathcal{N}_{\tau}\right):=Tx+\mathcal{N}_{\tau}$ for $x\in\mathscr{X}$. It is well-defined by $\langle Tx,Tx\rangle \leq \|T\|^2 \langle x,x\rangle$; see \cite[Proposition 1.2]{lance}. In addition, the operator $\phi(T)$ has a unique extension $\phi_{\tau}(T)\in\mathcal{L}(\mathscr{H}_{\tau})$ and the deduced map 
\begin{equation}\label{mos5}
\phi_{\tau}:\mathcal{L}(\mathscr{X})\rightarrow\mathcal{L}(\mathscr{H}_{\tau})
\end{equation}
is a $*$-homomorphism.
 
 Let $x_1,\dots,x_n\in\mathscr{X}$. We put
 \begin{equation*}
\left\|\left(x_1,\dots,x_n\right)\right\|_n^{\mathcal{P}\left(\mathfrak{A}\right)}:=\sup_{\tau\in\mathcal{P}\left(\mathfrak{A}\right)}\left\|\left(x_1+\mathcal{N}_{\tau},\dots,x_n+\mathcal{N}_{\tau}\right)\right\|_n^{\mathscr{H}_{\tau}}.
 \end{equation*}
 
It follows from \cite[Theorem 5.1.11]{mor} that
\begin{align*}
\sup_{\tau\in\mathcal{P}\left(\mathfrak{A}\right)}\left\|x+\mathcal{N}_{\tau}\right\|^{\mathscr{H}_{\tau}}=\sup_{\tau\in\mathcal{P}\left(\mathfrak{A}\right)}\tau\left( \langle x,x\rangle\right)^{1/2}=\left\|x\right\|.
\end{align*}
 It is straightforward to verify that $\left( \left(\mathscr{X}^n,\left\|\cdot\right\|_n^{\mathcal{P}\left(\mathfrak{A} \right)} \right):n\in\mathbb{N}\right)$ is a multi-normed space.
\begin{lemma}\label{orth}
If $x_1,x_2,\dots,y_n$ are mutually orthogonal elements in $\mathscr{X}$, then $$\mu_n^*(x_1,x_2,\dots,x_n)=\max_{1\leq i\leq n }\left\|x_i\right\|.$$
\end{lemma}
\begin{proof}
We have
\begin{align*}
\max_{1\leq i\leq n }\left\|x_i\right\|&\leq\mu_n^*(x_1,\dots,x_n)\tag{by \eqref{powernorm}}\\&=\sup\left\lbrace\left\| \sum_{i=1}^{n}x_ia_i\right\| : \left\|\sum_{i=1}^{n}a_i^*a_i \right\|\leq 1\right\rbrace\tag{by \eqref{Lemma3.7}}\\
&=	\sup\left\lbrace\left\| \sum_{i=1}^{n}a_i^*\left\langle x_i,x_i\right\rangle a_i\right\|^{1/2} : \left\|\sum_{i=1}^{n}a_i^*a_i \right\|\leq 1\right\rbrace\leq\max_{1\leq i\leq n }\left\|x_i\right\|.
\end{align*}
\end{proof}
In the following proposition, we introduce some inequalities between the norms $\left\|\cdot\right\|_n^{\mathscr{X}}$, $\left\|\cdot\right\|_n^*$, and $\left\|\cdot\right\|_n^{\mathcal{P}\left(\mathfrak{A} \right)}$ based on a Hilbert $\mathfrak{A}$-module $\mathscr{X}$.
\begin{proposition}\label{module}
Suppose that $\mathscr{X}$ is a Hilbert $\mathfrak{A}$-module. Let $x=\left(x_1,\dots,x_n \right)\in\mathscr{X}^n$. Then
\begin{enumerate}
\item
$\left\|x\right\|_n^{\mathscr{X}}\leq\left\|x\right\|_n^*$,
\item
$\left\|x\right\|_n^{\mathscr{X}}\leq\left\|x\right\|_n^{\mathcal{P}\left(\mathfrak{A} \right)}$,
\item
if $\mathfrak{A}$ is commutative and o-dim${\mathscr{H}_{\tau}}\geq n$, then 
$\left\|x\right\|_n^{*}\leq\left\|x\right\|_n^{\mathcal{P}\left(\mathfrak{A} \right)}$
for each $\tau\in\mathcal{P}\left(\mathfrak{A} \right)$.
\end{enumerate}
\end{proposition}
\begin{proof}
(1) Without loss of generality, we may assume that $\mathfrak{A}$ is unital. Let $x_1,\dots,x_n\in\mathscr{X}$ be given. Let $(P_i:1\leq i\leq n)$ be a family of mutually orthogonal projections summing to $I_{\mathscr{X}}$. Pick $\varepsilon>0$ and set $y_i=P_ix_i\left(\left|P_ix_i\right|^2+\varepsilon/n\right)^{-1/2}$ for $1\leq i\leq n$. Then
\begin{align*}
\left\langle y_i,y_i\right\rangle=	\left(\left|P_ix_i\right|^2+\varepsilon/n\right)^{-1/2}\left|P_ix_i\right|^2\left(\left|P_ix_i\right|^2+\varepsilon/n\right)^{-1/2}\leq1.
\end{align*}
Since $\left\langle y_i,y_j\right\rangle=0$ for $1\leq i\neq j\leq n$, Lemma \ref{orth} yields that $\mu_n^*(y_1,y_2,\dots,y_n)\leq 1$. Moreover, 
\begin{align*}
\left|P_ix_i\right|^2 &\leq\left|P_ix_i\right|^2\left(\left|P_ix_i\right|^2+\varepsilon/n\right)^{-1}\left|P_ix_i\right|^2+\varepsilon/n\\
&\leq\left| \left\langle P_ix_i\left(\left|P_ix_i\right|^2+\varepsilon/n\right)^{-1/2},Px_i\right\rangle\right| ^2+\varepsilon/n\\
&\leq\left| \left\langle y_i,x_i\right\rangle\right| ^2+\varepsilon/n.
\end{align*}
Employing the fact that if $a, b\in\mathfrak{A}$ and $0\leq a\leq b$, then $\|a\|\leq \|b\|$, we get
	\begin{equation*}
		\left\| \left|P_1x_1 \right|^2+\dots+\left|P_nx_n \right|^2\right\|\leq\left\|\sum_{i=1}^{n}\left|\left\langle y_i,x_i\right\rangle\right| ^{2}\right\|+\varepsilon. 
	\end{equation*}
	Therefore, we arrive at $\left\|x \right\|_n^{\mathscr{X}} \leq \left\|x\right\|_n^*$.
	
	(2) Pick $(P_i:1\leq i\leq n)$ be a family of mutually orthogonal projections summing to $I_{\mathscr{X}}$. It derives from \cite[Theorem 5.1.11]{mor} that there exists $\tau\in\mathcal{P}\left(\mathfrak{A}\right)$ such that $\left\|\sum_{i=1}^{n}\left| P_ix_i\right|^2\right\|=\tau\left( \sum_{i=1}^{n}\left| P_ix_i\right|^2\right)$. Since the map $\phi_{\tau}:\mathcal{L}(\mathscr{X})\rightarrow\mathcal{L}(\mathscr{H}_{\tau})$ is a $*$-homomorphism, we observe that $(\phi_{\tau}(P_i):1\leq i\leq n)$ is a family of mutually orthogonal projections in $\mathcal{L}(\mathscr{H}_{\tau})$. Thus, we arrive at 
	\begin{align*}
\left\|\sum_{i=1}^{n} P_ix_i\right\|^2&=\left\|\sum_{i=1}^{n}\left| P_ix_i\right|^2\right\|= \sum_{i=1}^{n}\tau\left( \left| P_ix_i\right|^2\right)=\sum_{i=1}^{n} \left\langle P_ix_i+\mathcal{N}_{\tau}, P_ix_i+\mathcal{N}_{\tau}\right\rangle \\
&=\sum_{i=1}^{n} \left| P_ix_i+\mathcal{N}_{\tau}\right|^2 = \sum_{i=1}^{n}\left|\phi_{\tau}(P_i)\left( x_i+\mathcal{N}_{\tau}\right) \right|^2 \\
&=\left| \sum_{i=1}^{n} \phi_{\tau}(P_i)( x_i+\mathcal{N}_{\tau}) \right|^2 \quad ({\rm by~the~orthogonality~of~\phi_{\tau}(P_i)'s})\\
&\leq\left(\left\|\left(x_1+\mathcal{N}_{\tau},\dots,x_n+\mathcal{N}_{\tau}\right)\right\|_n^{\mathscr{H}_{\tau}}\right)^2\quad ({\rm by~\eqref{Hilbert multi norm}}).
	\end{align*}
	Hence $\left\|x\right\|_n^{\mathscr{X}}\leq\left\|x\right\|_n^{\mathcal{P}\left(\mathfrak{A} \right)}$.
	
	(3) Fix $n\in\mathbb{N}$ and $x=\left( x_1,\dots,x_n\right) \in\mathscr{X}^n$. Let $\left( y_1,\dots,y_n\right) \in\mathscr{X}^n$ with $\mu^*_n(y_1,\dots,y_n)\leq1$. There exists $\tau\in\mathcal{P}\left(\mathfrak{A}\right)$ such that $\tau\left(\left( \sum_{i=1}^{n}\left|\left\langle y_i,x_i \right\rangle\right|^2 \right)^{{1}/{2}}\right)=\left\|\left( \sum_{i=1}^{n}\left|\left\langle y_i,x_i \right\rangle\right|^2 \right)^{{1}/{2}}\right\|$. It follows from \eqref{Proposition 3.6} that
	\begin{equation}\label{min}
			\mu^*_{n}(y_1,\ldots,y_n)=\min\left\lbrace\lambda>0:\left( \sum_{i=1}^{n}\left| \left\langle y_i,x\right\rangle\right|^2\right)^{1/2}\leq\lambda\left|x\right|\,\ \text{for all $x\in\mathscr{X}$}\right\rbrace.
		\end{equation}
	Therefore, $\sum_{i=1}^{n}\left| \left\langle y_i,x\right\rangle\right|^2\leq\left\langle x,x\right\rangle$ for all $x\in\mathscr{E}$. \cite[Theorem 3.3.2]{mor} shows that $\tau$ is a $*$-homomorphism. Hence,
	\begin{equation}\label{tau}
\sum_{i=1}^{n}\left|\left\langle y_i+\mathcal{N}_{\tau},x+\mathcal{N}_{\tau}\right\rangle\right|^2=\tau\left(\sum_{i=1}^{n}\left| \left\langle y_i,x\right\rangle\right|^2\right)\leq\tau\left( \left\langle x,x\right\rangle\right) =\left|x+\mathcal{N}_{\tau}\right|^2
	\end{equation}
for all $x\in\mathscr{X}$. Since $\mathscr{H}_{\tau}$ is the Hilbert completion of $\mathscr{E}/\mathcal{N}_{\tau}$, we derive from \eqref{min} and \eqref{tau} that $\mu^{\mathscr{H}_{\tau}}_n\left( y_1+\mathcal{N}_{\tau},\ldots,y_n+\mathcal{N}_{\tau}\right)\leq 1$, where $\mu^{\mathscr{H}_{\tau}}_n$ is calculated with respect to the Hilbert space $\mathscr{H}_{\tau}$. Thus,
	\begin{align*}
		\left\|\left( \sum_{i=1}^{n}\left|\left\langle y_i,x_i \right\rangle\right|^2 \right)^{{1}/{2}}\right\|&=\tau\left(\left(\sum_{i=1}^{n}\left|\left\langle y_i,x_i\right\rangle\right|^2 \right)^{{1}/{2}}\right)\\&=\left(\sum_{i=1}^{n}\left|\tau\left\langle y_i,x_i\right\rangle\right|^2 \right)^{{1}/{2}}\tag{$\tau$ is a $*$-homomorphism}\\&=\left(\sum_{i=1}^{n}\left|\left\langle y_i+\mathcal{N}_{\tau},x_i+\mathcal{N}_{\tau}\right\rangle\right|^2 \right)^{{1}/{2}}\\& \leq\left\|\left(x_1+\mathcal{N}_{\tau},\dots,x_n+\mathcal{N}_{\tau} \right) \right\|_n^*\qquad\qquad\qquad\qquad\qquad({\rm by~ \eqref{mos3}})\\& =\left\|\left(x_1+\mathcal{N}_{\tau},\dots,x_n+\mathcal{N}_{\tau} \right) \right\|_n^{\mathscr{H}_{\tau}}\tag{by the fact that o-dim${\mathscr{H}_{\tau}}\geq n$ and \eqref{dim} }.
	\end{align*}
	Hence $\left\|x\right\|_n^*\leq\left\|x\right\|_n^{\mathcal{P}\left(\mathfrak{A} \right)}$.
\end{proof}
In the following examples, we calculate the multi-norms in the case when $\mathscr{X}=\mathit{l}_n^2(\mathscr{\mathfrak{A}})$ for a commutative $C^*$-algebra $\mathfrak{A}$. 

\begin{example}\label{example}
Let $\mathscr{X}=\mathit{l}_n^2(\mathscr{\mathfrak{A}})$, where $\mathfrak{A}$ is a unital commutative $C^*$-algebra. Then $\left\|x\right\|_m^{\mathscr{X}}=\left\|x\right\|_m^{*}=\left\|x\right\|_m^{\mathcal{P}\left(\mathfrak{A} \right)}$ , where $m\leq n$ and $x=\left(x_1,\dots,x_m\right)\in\mathscr{X}^m$. To prove this, let $\tau\in\mathcal{P}(\mathfrak{A})$. It follows from \cite[Theorem 2.1.10]{mor} that $\mathfrak{A}=C(\Omega)$, where $\Omega$ is a compact Hausdorff space. It follows from \cite[Theorems 2.1.15 and 5.1.6]{mor} that there exists $\omega\in\Omega$ such that $\tau(f)=f(\omega)$ for $f\in C(\Omega)$. Thus
\begin{align*}
	\mathcal{N}_{\tau}&=\left\lbrace x\in \mathit{l}_n^2(C(\Omega)):\tau(\left\langle x,x\right\rangle)=0 \right\rbrace\\
	&=\left\lbrace (f_i)\in \mathit{l}_n^2(C(\Omega)):\tau\left(\sum_{i=1}^{n}\left|f_i\right|^2\right)=0 \right\rbrace\\& =\left\lbrace (f_i)\in \mathit{l}_n^2(C(\Omega)):f_i(\omega)=0 \mbox{~for~all~} 1\leq i \leq n \right\rbrace.
\end{align*}
We derive that the map $\mathscr{X}/\mathcal{N}_{\tau}\rightarrow \mathit{l}_n^2$ defined by $(f_i)+\mathcal{N}_{\tau}\mapsto(f_i(w))$ is an isometry and surjective. So $\mathscr{H}_{\tau}=\mathit{l}_n^2$. Let $(p_i:1\leq i\leq m)$ be a family of mutually orthogonal projections in $\mathcal{L}(\mathscr{H}_{\tau})=\mathcal{L}(\mathit{l}_n^2)=\mathbb{M}_n(\mathbb{C})$ summing to $I_{\mathit{l}_n^2}$. It follows from \cite[p. 16]{Manu} that $\mathcal{L}\left( \mathscr{X} \right)=\mathbb{M}_n(\mathfrak{A})$. In addition, we identify $\mathbb{M}_n(\mathbb{C})$ with $\mathbb{M}_n(\mathbb{C}) I_{\mathbb{M}_n(\mathfrak{A})}$. Let $x_i=(f_{ij}:1\leq j \leq n)$ for $1\leq i \leq m$, where $f_{ij}\in C(\Omega)$ for all $i$ and $j$. Then
\begin{align*}
	\sup_{\tau\in\mathcal{P}\left(\mathfrak{A}\right)}\left\|\sum_{i=1}^{m}p_i(x_i+\mathcal{N}_{\tau})\right\|&=\sup_{\omega\in\Omega}\left\|\sum_{i=1}^{m}p_i(f_{ij}(\omega))_j\right\| =\sup_{\omega\in\Omega}\left\|\left( \sum_{i=1}^{m} p_i\left( f_{ij}\right)_j \right) (\omega) \right\|\\&= \left\| \sum_{i=1}^{m} p_i\left( f_{ij}\right)_j \right\|= \left\| \sum_{i=1}^{m} p_ix_i \right\|.
\end{align*}
Hence $\left\|x\right\|_n^{\mathcal{P}\left(\mathfrak{A} \right)}\leq\left\|x\right\|_n^{\mathscr{X}}$. Proposition \ref{module} entails that $\left\|x\right\|_n^{\mathscr{X}}=\left\|x\right\|_n^{*}=\left\|x\right\|_n^{\mathcal{P}\left(\mathfrak{A} \right)}$.
\end{example}
\begin{example}
Assume $\mathscr{X}=\mathfrak{A}$ as a Hilbert $\mathfrak{A}$-module, where $\mathfrak{A}$ is commutative. We show that $\left\|x\right\|_n^{\mathscr{X}}=\left\|x\right\|_n^{*}=\left\|x\right\|_n^{\mathcal{P}\left(\mathfrak{A} \right)}=\max_{1\leq i\leq n }\left\|a_i\right\|$, where $a=\left(a_1,\dots,a_n\right)\in\mathfrak{A}^n$. To this end, suppose that $\mathfrak{A}$ is unital. We derive from Example \ref{example} that $\mathscr{H}_{\tau}=\mathbb{C}$.
 The minimum multi-norm is the unique multi-norm based on $\mathbb{C}$ \cite[Proposition 3.6]{2012}. Thus,
\begin{align*}
\max_{1\leq i\leq n }\left\|a_i\right\|\leq\left\|\left(a_1,\dots,a_n\right)\right\|_n^{\mathcal{P}\left(\mathfrak{A} \right)}&=\sup_{\tau\in\mathcal{P}\left(\mathfrak{A}\right)}\left\|\left(a_1+\mathcal{N}_{\tau},\dots,a_n+\mathcal{N}_{\tau}\right)\right\|_n^{\mathscr{H}_{\tau}}\\&=\sup_{\tau\in\mathcal{P}\left(\mathfrak{A}\right)}\max_{1\leq i\leq n }\left\|a_i+\mathcal{N}_{\tau}\right\|\leq\max_{1\leq i\leq n }\left\|a_i\right\|.
\end{align*}
In the case where $\mathfrak{A}$ has no unit, let $\mathfrak{B}$ be the minimal unitization $\mathfrak{A}\oplus \mathbb{C}$ of $\mathfrak{A}$. Pick $\tau\in\mathcal{P}\left(\mathfrak{A}\right)$. It follows from \cite[Theorem 5.1.13]{mor} that there is a unique pure state $\tilde{\tau}$ on $\mathfrak{B}$ extending $\tau$. Obviously, the map $\mathfrak{A}/\mathcal{N}_{\tau}\rightarrow\mathfrak{B}/\mathcal{N}_{\tilde{\tau}}$ defined by $x+\mathcal{N}_{\tau}\mapsto x+\mathcal{N}_{\tilde{\tau}}$ is an isometry So $\mathscr{H}_{\tau}$ is a nonzero subspace of ${H}_{\tilde{\tau}}$. Thus, $\mathscr{H}_{\tau}=\mathscr{H}_{\tilde{\tau}}=\mathbb{C}$. Therefore, $\left\|a\right\|_n^{\mathcal{P}\left(\mathfrak{A} \right)}=\max_{1\leq i\leq n }\left\|a_i\right\|$. 
	
Furthermore, it immediately follows from Proposition \ref{module} that $\left\|a\right\|_n^{\mathscr{X}}=\left\|a\right\|_n^{*}=\left\|a\right\|_n^{\mathcal{P}\left(\mathfrak{A} \right)}=\max_{1\leq i\leq n }\left\|a_i\right\|$.
\end{example}
 
 \section{Multi-norms based on a Hilbert $\mathbb{K}\left(\mathscr{H}\right)$-modules}
 
 Following \cite[Definition 1]{compact}, an element $u \in \mathscr{X}$ is said to be a basic vector if
 $e=\left\langle u, u \right\rangle$ is a \emph{minimal projection} in $\mathfrak{A}$ in the sense that $e\mathfrak{A} e= \mathbb{C}e$. Every minimal projection in $\mathcal{L}(\mathscr{H})$ has the form
 $\theta_{\xi,\xi}$ for some $\xi \in \mathscr{H}_{[1]}$. A system 	$(u_{\lambda})_{\lambda \in \Lambda}$
 in	$E$ is \emph{orthonormal} if each $u_{\lambda}$ is a basic vector and $\left\langle u_{\lambda}, u_{\mu} \right\rangle=0$ for all $\lambda \neq \mu$.	An \emph{orthonormal system} 	$(u_{\lambda})$ in $\mathscr{X}$ is called an \emph{orthonormal basis} if it generates a dense submodule of	$\mathscr{X}$. Every Hilbert 	$C^*$-module over a	$C^*$-algebra $\mathfrak{A}$ of compact operators admits an orthonormal basis. It follows that every closed submodule of $\mathscr{X}$ is complemented; see \cite[Theorem 4]{compact}\label{24} and also \cite{ARA, ASA}. All orthonormal bases have the same cardinality; see \cite[Proposition 1.11]{cabrera}. The orthogonal dimension of $\mathscr{X}$, in short \emph{o-dim$(\mathscr{X})$}, is defined as the cardinal number of any one of its
 orthonormal bases.
 
To achieve the next result, we need the following lemma.
 
 \begin{lemma}\cite[Theorem 1]{compact}\label{compact}
 	Let $(u_{\lambda})_{\lambda \in \Lambda}$
 	be an orthonormal system in 
 	$\mathscr{X}$.
 	The following statements are mutually equivalent:
 	\begin{enumerate}
 		\item
 		$(u_{\lambda})$
 		is an orthonormal basis for 
 		$\mathscr{X}$.
 		\item
 		$x=\sum_{\lambda \in \Lambda}u_{\lambda}\left\langle u_{\lambda},x\right\rangle$ for every
 		$x \in \mathscr{X}$.
 	\end{enumerate}
 \end{lemma}
 \begin{lemma}\label{base}
 	Let $\mathscr{X}$ be a Hilbert $\mathbb{K}\left(\mathscr{H}\right)$-module with an orthonormal basis $\left(u_{\lambda}\right)_{\lambda \in \Lambda} $, and let $\left( \eta_{\lambda}\right)_{\lambda\in\Lambda}$ be an arbitrary net of unit vectors in $\mathscr{H}$. Then there exists an orthonormal basis $\left(v_{\lambda}\right)_{\lambda \in \Lambda} $ for $\mathscr{X}$	such that 	$\left\langle v_{\lambda},v_{\lambda}\right\rangle = \theta_{\eta_{\lambda},\eta_{\lambda}}$ for each	$ \lambda \in \Lambda$.
 \end{lemma}
 \begin{proof}
Suppose that $\left\langle u_\lambda, u_\lambda\right\rangle = \theta_{\xi_\lambda,\xi_\lambda}$,	where 	$\xi_\lambda\in\mathscr{H}$ is a unit vector in $\mathscr{H}$ for each	$ \lambda \in \Lambda$. 	 
Set $v_\lambda=u_\lambda \theta_{\xi_\lambda,\eta_\lambda}$.
 	Then $\left(v_{\lambda}\right)_{\lambda \in \Lambda} $ is an orthonormal system for $\mathscr{X}$ with $\left\langle v_{\lambda},v_{\lambda}\right\rangle = \theta_{\eta_{\lambda},\eta_{\lambda}}$ for each	$\lambda \in \Lambda$. We observe that
 	\begin{align*}
 		\sum_{\lambda \in \Lambda}v_{\lambda}\left\langle v_{\lambda},x\right\rangle&=\sum_{\lambda \in \Lambda}u_\lambda\theta_{\xi_\lambda,\eta_{\lambda}}\left\langle u_\lambda\theta_{\xi_\lambda,\eta_{\lambda}},x\right\rangle 
 		 =\sum_{\lambda \in \Lambda}u_\lambda \theta_{\xi_\lambda,\xi_\lambda}\left\langle u_\lambda,x\right\rangle\tag{$\theta_{\xi_\lambda,\eta_{\lambda}}\theta_{\eta_{\lambda},\xi_\lambda}=\theta_{\xi_\lambda,\xi_\lambda}$}\\ 
 		&= \sum_{\lambda \in \Lambda}u_{\lambda}\left\langle u_{\lambda},x\right\rangle\tag{$u_\lambda \theta_{\xi_\lambda,\xi_\lambda}=u_\lambda$}\\&=x\tag{by Lemma \ref{compact}}
 	\end{align*}
 	for each $x\in\mathscr{X}$. It deduces from Lemma \ref{compact} that
 	$\left(v_\lambda\right)_{\lambda\in\Lambda}$ is an orthonormal basis
 	such that $\left\langle v_{\lambda}, v_{\lambda}\right\rangle = \theta_{\eta_{\lambda},\eta_{\lambda}}$.
 \end{proof} 
The \emph{(absolutely) convex hull} of a nonempty subset $\mathbb{S}$ of a normed space $\mathcal{X}$ is defined as follows:
\begin{align*}
{\rm co}(\mathbb{S}) := \left\{ \sum_{i=1}^{n}t_ix_i : t_1, \dots,t_n \in \mathbb{R}^+,\, \sum_{i=1}^nt_i =1,\, x_1, \dots,x_n \in \mathbb{S}\right\},\\
\left( {\rm aco}(\mathbb{S}): = \left\lbrace \sum_{i=1}^{n}\alpha_ix_i : \alpha_1,\dots,\alpha_n\in\mathbb{C}
\sum_{i=1}^n\left| \alpha_i\right| \leq 1,\, x_1, \dots,x_n \in \mathbb{S}\right\rbrace\right).
 \end{align*}
 If $\alpha\mathbb{S} \subseteq \mathbb{S}$ for each $\alpha \in \mathbb{C}_{[1]}$, then $\mathbb{S}$ is called \emph{ balanced}. In the case where $\mathbb{S}$ is balanced, ${\rm aco}(\mathbb{S}) = {\rm co}(\mathbb{S})$. In the following lemma, we apply the Russo--Dye theorem, asserting that if $\mathfrak{A}$ is a $C^*$-algebra, then $\mathfrak{A}_{[1]}=\overline{aco}(\mathcal{U}(\mathfrak{A}))$, where $\mathcal{U}(\mathfrak{A})$ stands for the unitary group of $\mathfrak{A}$ and the closure is taken with respect to the norm topology. In what follows, we modify the notion of
 ${\rm aco}(\mathbb{S})$ for our investigation. Let $\mathscr{X}$ be a Hilbert $\mathfrak{A}$-module and let $\mathbb{S}\subseteq\mathscr{X}$. We set
 \begin{equation*}
 	{\rm aco}_{\mathfrak{A}}(\mathbb{S}):=\left\lbrace\sum_{i=1}^{n}x_ia_i : a_1,\dots,a_n\in\mathfrak{A},\, 
 	\sum_{i=1}^n\left\|a_i\right\|\leq 1,\, x_1, \dots,x_n \in \mathbb{S} \right\rbrace. 
 \end{equation*}
 The closure of ${\rm aco}_{\mathfrak{A}}(\mathbb{S})$ is taken with respect to the norm topology. Employing an approximate identity of $\mathfrak{A}$, it follows immediately that $ {\rm aco}(\mathbb{S})\subseteq\overline{aco}_{\mathfrak{A}}(\mathbb{S})$. The family of all mutually orthogonal $n$-tuples $\left(x_1,\dots,x_n\right)\subseteq\mathit{l}_n^2(\mathscr{X})$, where $\left\langle x_i,x_i\right\rangle$ is a projection for $1\leq i\leq n$, is denoted by $\mathcal{D}^{\mathscr{X}}_n$. 

\begin{lemma}\label{CO}
 Let $\mathscr{X}$ be a Hilbert $\mathfrak{A}$-module. Then
 \begin{enumerate}
 	\item 
 	$\overline{\rm aco}_{\mathfrak{A}}(\mathcal{D}^{\mathscr{X}}_n)\subseteq\left( \mathscr{X}^n,\mu^*_n\right)_{[1]}$;
 	\item
 	in the case when $\mathscr{X}$ is a Hilbert $\mathbb{K}\left(\mathscr{H}\right)$-module, where $\mathscr{H}$ is separable and $\text{o-dim}\left(\mathscr{X}\right)\geq n\times\left( \text{o-dim}\left(\mathscr{H}\right)\right)$, it holds that $\overline{\rm aco}_{\mathfrak{A}}(\mathcal{D}^{\mathscr{X}}_n)=\left( \mathscr{X}^n, \mu^*_n\right)_{[1]}$.
 \end{enumerate}
\end{lemma}
\begin{proof} 
	(1) 	Lemma \ref{orth} ensures that $\mathcal{D}^{\mathscr{X}}_n\subseteq \left( \mathscr{X}^n, \mu^*_n\right)_{[1]}$. Let $a_1,\dots,a_m\in\mathfrak{A}$ with $\sum_{i=1}^{m}\left\| a_i\right\| \leq1$. We have
	\begin{align*}
		\mu^*_n\left(\sum_{i=1}^{m}x_ia_i \right)\leq \sum_{i=1}^{m}\left\| a_i\right\| \mu^*_n\left( x_i\right)\leq1\qquad\left(x_i\in\mathcal{D}^{\mathscr{X}}_n \right) , 
	\end{align*}
where the first inequality is immediately deduced from the definition of $\mu^*$ \eqref{mos4}. Thus, 	${\rm aco}_{\mathfrak{A}}(\mathcal{D}^{\mathscr{X}}_n)\subseteq\left( \mathscr{X}^n,\mu^*_n\right)_{[1]}$. 

Let $\left(y^{(k)}:k\in\mathbb{N} \right) $ be a sequence in $\left( \mathscr{X}^n,\mu^*_n\right)_{[1]}$ such that $y^{(k)}=(y^{(k)}_1,\dots,y^{(k)}_n)\rightarrow y=(y_1,\dots,y_n)$ in the norm topology on $\mathit{l}^2_n\left(\mathscr{X}\right) $. So $\lim_{k\rightarrow\infty}\sum_{i=1}^{n}\left\|y^{(k)}_i-y_i\right\|=0$. It follows from $\eqref{powernorm}$ that
$\mu_n^*(y^{(k)})\rightarrow\mu_n^*(y)$. Hence, $\left( \mathscr{X}^n,\mu^*_n\right)_{[1]}$ is closed with respect to the norm topology on $\mathit{l}^2_n\left(\mathscr{X} \right)$. It follows that 	$\overline{\rm aco}_{\mathfrak{A}}(\mathcal{D}^{\mathscr{X}}_n)\subseteq\left( \mathscr{X}^n,\mu^*_n\right)_{[1]}$.

	(2)
	Take $y =(y_1,\dots,y_n)\in \left( \mathscr{X}^n,\mu^*_n\right)_{[1]}$. Let $(\xi_j)$ be an orthonormal basis for $\mathscr{H}$. Take $m\in\mathbb{N}$. We utilize the structure in the proof of Lemma \ref{base} to construct an orthonormal system $\left( u_{ij}: 1\leq i \leq n , 1\leq j \leq m\right)$, where $\left\langle u_{ij},u_{ij}\right\rangle=\theta_{\xi_j,\xi_j}$ for all $i, j$. The hypothesis on the dimension of $\mathscr{X}$ allows us to have $n$ elements $u_{ij},\,\, i=1, \dots, n$, for each $\xi_j$. Now, let us define $T\in \mathcal{L}\left( \mathscr{X}\right)$ by
	$$T:=\sum_{i,j}\theta_{y_i,u_{ij}}.$$
 We claim that $T\in \mathcal{L}\left(\mathscr{X}\right)_{[1]}$. To this end, let $x\in\mathscr{X}_{[1]}$. It deduces from Lemma \ref{compact} that 
	\begin{equation}\label{x}
\left\| \sum_{i,j}\left\langle x,u_{ij}\right\rangle\left\langle u_{ij},x\right\rangle\right\|\leq\|\langle x,x \rangle\|= \left\| x\right\|^2\leq 1. 
\end{equation}
We derive from \eqref{Lemma3.7} and \eqref{x} that
\begin{align*}
\left\| Tx\right\|^2=	\left\| \sum_{i,j}Tu_{ij}\left\langle u_{ij},x\right\rangle\right\|&\leq\mu^*_n\left(Tu_{11},\dots,Tu_{nm} \right)\\&=\mu^*_n\left(y_1\theta_{\xi_1,\xi_1},\dots,y_1\theta_{\xi_m,\xi_m},\dots, y_n\theta_{\xi_1,\xi_1},\dots,y_n\theta_{\xi_m,\xi_m}\right)\\&=\sup_{z\in\mathscr{X}_{[1]}}\left\|\sum_{i=1}^{n}\sum_{j=1}^{m}\left\langle z,y_i\right\rangle\theta_{\xi_j,\xi_j}\left\langle y_i,z\right\rangle \right\|\\&\leq\sup_{z\in\mathscr{X}_{[1]}}\left\|\sum_{i=1}^{n}\left\langle z,y_i\right\rangle\left\langle y_i,z\right\rangle \right\|\tag{$\sum_{j=1}^{m}\theta_{\xi_j,\xi_j}$ is a projection}\\&\leq\mu^*_n\left(y_1,\dots,y_n\right)\leq1 . 
\end{align*}
 We see that $T\in \mathcal{L}\left(\mathscr{X}\right)_{[1]}$, and so $T\in \overline{\rm co}({\mathcal U}\left( \mathcal{L}\left( \mathscr{X}\right) \right) $ by the Russo--Dye theorem. Set $v_i=\sum_{j=1}^{m}u_{ij}$ and $a_m=\sum_{j=1}^{m}\theta_{\xi_j,\xi_j}$. Since $\left\langle v_i,v_i\right\rangle=\sum_{j=1}^{m}\left\langle u_{ij},u_{ij}\right\rangle=a_m$, we have $(v_1,\dots,v_n)\in\mathcal{D}_n^{\mathscr{X}}$. In addition, $T\left( v_i \right)=\sum_{j=1}^{m}\theta_{y_i,u_{ij}}u_{ij}=y_i\sum_{j=1}^{m}\left\langle u_{ij},u_{ij}\right\rangle=y_ia_m$. There exists a sequence $\left( T_k: k\in\mathbb{N}\right) $ in ${\rm co}({\mathcal U}\left( \mathcal{L}\left( \mathscr{X}\right) \right)$ such that $T_k\rightarrow T$. Evidently, $(Uv_1,\dots,Uv_n)\in\mathcal{D}_n^{\mathscr{X}}$ for each $U\in{\mathcal U}\left( \mathcal{L}\left( \mathscr{X}\right) \right)$. It follows that 
	\begin{equation}\label{k}
 \left( T_k\left( v_1 \right) ,\dots,T_k\left( v_n\right) \right)\in {\rm co}(\mathcal{D}^{\mathscr{X}}_n).
\end{equation}
Moreover, $ {\rm co}(\mathcal{D}^{\mathscr{X}}_n) 
 ={\rm aco}(\mathcal{D}^{\mathscr{X}}_n)$ since $\mathcal{D}_n^{\mathscr{X}}$ is a balanced subset of $\mathit{l}^2_n\left( \mathscr{X}\right)$. Furthermore, as we mentioned above, ${\rm aco}(\mathcal{D}^{\mathscr{X}}_n)\subseteq\overline{\rm aco}_{\mathfrak{A}}(\mathcal{D}^{\mathscr{X}}_n)$. Hence
 	\begin{align*}
 	\left( T_k\left( v_1 \right) ,\dots,T_k\left( v_n\right) \right)\in \overline{\rm aco}_{\mathfrak{A}}(\mathcal{D}^{\mathscr{X}}_n)\tag{by \eqref{k}}.
 \end{align*}
It deduces that
	\begin{align*}
\left(y_1a_m,\dots,y_na_m\right)=\left( T\left( v_1 \right) ,\dots,T\left( v_n\right) \right)=\lim_{k\to\infty}\left( T_k\left( v_1 \right) ,\dots,T_k\left( v_n\right) \right)\in \overline{\rm aco}_{\mathfrak{A}}(\mathcal{D}^{\mathscr{X}}_n).
\end{align*} 
 Since $\left(a_m: m\in\mathbb{N} \right) $ is an approximate unit for $\mathbb{K}(\mathscr{H})$ \cite[p. 78]{mor}, we derive that
	\begin{align*}
\left(y_1,\dots,y_n\right)=\lim_{m\to\infty}\left(y_1a_m,\dots,y_na_m\right)\in\overline{\rm aco}_{\mathfrak{A}}\left( \mathcal{D}^{\mathscr{X}}_n\right).
\end{align*}
\end{proof}
 \begin{lemma}\label{multinorm}
 	Suppose that $\mathscr{X}$ is a Hilbert $\mathfrak{A}$-module. Let $x=\left( x_1,\dots,x_n\right) \in\mathscr{X}^n$. Then
 \begin{equation}\label{orthog}
 \sup\left\| \sum_{i=1}^{n} \left\langle x_i,v_i\right\rangle a_i\right\|\leq	\left\|x\right\|^{\mathscr{X}}_n,
 \end{equation}
 where the supremum is taken over all $(a_1,\dots,a_n) \in (\mathit{l}\left( \mathfrak{A}\right)^2_n)_{[1]}$ and $\left(v_1,\dots,v_n \right)\in \mathcal{D}_n^\mathscr{X}$.
 \end{lemma}
 \begin{proof}
 	 Let $\left( v_1,\dots,v_n\right)\in\mathcal{D}^{\mathscr{X}}_n $ such that $\left\langle v_i,v_i\right\rangle$ is a projection for $1\leq i\leq n$. Then $\left\langle v_i\left\langle v_i,v_i\right\rangle-v_i,v_i\left\langle v_i,v_i\right\rangle-v_i\right\rangle=0$. So $v_i\left\langle v_i,v_i\right\rangle=v_i$. For $1\leq i\leq n-1$, set $P_i=\theta_{v_i,v_i}$. Suppose that $P_n=I_{\mathscr{X}}-\sum_{i=1}^{n-1}\theta_{v_i,v_i}$. Then it immediately follows that $\theta_{v_n,v_n}\leq P_n$ and that $\left( P_i:1\leq i \leq n\right)$ is the family of orthogonal projections summing to $I_{\mathscr{X}}$. Therefore,
 	\begin{align*}
 	\left| \left\langle v_i,x\right\rangle\right|^2= \left\langle x,v_i\right\rangle	\left\langle v_i,x\right\rangle\leq\left\langle P_ix,P_ix \right\rangle\quad (1\leq i\leq n).
 	\end{align*}
 Take $(a_1,\dots,a_n) \in (\mathit{l}\left( \mathfrak{A}\right)^2_n)_{[1]}$. It follows from the Cauchy--Schwarz inequality that 
 		\begin{align*}
 		\left\| \sum_{i=1}^{n} \left\langle x_i,v_i\right\rangle a_i\right\|&=\left\|\left\langle \left(\left\langle v_1,x_1 \right\rangle,\dots,\left\langle v_n,x_n \right\rangle\right) ,\left(a_1,\dots,a_n\right) 
 		\right\rangle_{\mathit{l}^2_n\left( \mathfrak{A}\right)} \right\|\\& \leq 	\left\| \sum_{i=1}^{n} \left\langle x_i,v_i\right\rangle\left\langle v_1,x_1 \right\rangle\right\|\leq\left\| \sum_{i=1}^{n}\left\langle P_ix,P_ix \right\rangle\right\|\leq\left\|x\right\|^{\mathscr{X}}_n.
 	\end{align*}
 	
 \end{proof}
\begin{lemma}\label{isomorphism}
Let $\mathscr{X}$ be a Hilbert $\mathbb{K}\left(\mathscr{H}\right)$-module. If $\tau\in\mathcal{P}\left(\mathbb{K}\left(\mathscr{H}\right)\right)$, then the map 	$\phi_{\tau}:\mathcal{L}(\mathscr{X})\rightarrow\mathcal{L}(\mathscr{H}_{\tau})$ constructed at \eqref{mos5} is an isometric $*$-isomorphism.
\end{lemma}
\begin{proof}
It follows from \cite[Example 5.1.1]{mor} that 
\begin{equation}
\mathcal{P}\left(\mathbb{K}\left(\mathscr{H}\right)\right)=\left\lbrace\tau_x:\mathbb{K}\left(\mathscr{H}\right)\rightarrow\mathbb{C}, u\mapsto\left\langle ux,x\right\rangle {\rm ~for~some~unit~vector~} x\right\rbrace .
\end{equation}
 So there exists a unit vector $e$ such that $\tau=\tau_e$. It deduces from \cite[Remark 4]{compact} that the map $\mathscr{X}/\mathcal{N_{\tau}}\rightarrow\mathscr{X}\theta_{e,e}$ defined by $x+N_{\tau}\mapsto x\theta_{e,e}$ is an isometric isomorphism. So 
 $\mathscr{H}_{\tau}=\mathscr{X}/\mathcal{N_{\tau}}$. We derive from \cite[Theorem 5]{compact} that the map 	$\phi_{\tau}:\mathcal{L}(\mathscr{X})\rightarrow\mathcal{L}(\mathscr{H}_{\tau})$ is an isometric $*$-isomorphism.
\end{proof}
 
\begin{theorem}
Suppose that $\mathscr{X}$ is a Hilbert $\mathbb{K}\left(\mathscr{H}\right)$-module and that $x=\left(x_1,\dots,x_n \right)\in\mathscr{X}^n$. Then
\begin{enumerate}
	\item 
	$\left\|x\right\|_n^{\mathscr{X}}=\left\|x\right\|_n^{\mathcal{P}\left(\mathfrak{A} \right)}$; 
	\item
	if $\mathscr{H}$ is separable and $\text{o-dim}\left(\mathscr{X}\right)\geq n\times\left( \text{o-dim}\left(\mathscr{H}\right)\right)$, then	$\left\|x\right\|_n^{\mathscr{X}}=\left\|x\right\|_n^{*}$.
\end{enumerate}
\end{theorem}
\begin{proof}
(1) Suppose that $\tau\in\mathcal{P}\left(\mathbb{K}\left(\mathscr{H}\right)\right)$. It deduces from Lemma \ref{isomorphism} that a family of mutually orthogonal projections in $\mathcal{L}(\mathscr{H}_{\tau})$ summing to $I_{\mathscr{H}_{\tau}}$ has the form $(\phi_{\tau}(P_i):1\leq i\leq n)$, where $(P_i:1\leq i\leq n)$ is a family of mutually orthogonal projections summing to $I_{\mathscr{X}}$. Thus, 
\begin{align*}
\left\|\sum_{i=1}^{n}\phi_{\tau}(P_i)(x_i+\mathcal{N}_{\tau})\right\|=\left\|\sum_{i=1}^{n}P_ix_i+\mathcal{N}_{\tau}\right\|\leq\left\|\sum_{i=1}^{n}P_ix_i\right\|.
\end{align*}
Therefore, $\left\|x\right\|_n^{\mathcal{P}\left(\mathfrak{A} \right)}\leq\left\|x\right\|_n^{\mathscr{X}}$, and we conclude from Proposition \ref{module}(2) that $\left\|x\right\|_n^{\mathscr{X}}=\left\|x\right\|_n^{\mathcal{P}\left(\mathfrak{A} \right)}$.

(2) Let $x =(x_1,\dots,x_n)\in \mathscr{X}^n$. Suppose that $y =(y_1,\dots,y_n)\in \mathscr{X}^n$ with $\mu_n^*(y_1,\dots,y_n)\leq1$. It follows from Lemma \ref{CO}(2) that $y\in\left( \mathscr{X}^n, \mu^*_n\right)_{[1]}=\overline{\rm aco}_{\mathfrak{A}}(\mathcal{D}^{\mathscr{X}}_n)$. Then
\begin{equation}\label{dence}
\left\|x\right\|^*_n=\sup\left\lbrace \left\|\sum_{i=1}^n\left\langle x_i,y_i\right\rangle\left\langle y_i,x_i\right\rangle\right\|^{1/2}:y\in{\rm aco}_{\mathfrak{A}}(\mathcal{D}^{\mathscr{X}}_n)\right\rbrace.
\end{equation}
Let $\sum_{j=1}^{m}z_ja_j\in {\rm aco}_{\mathfrak{A}}(\mathcal{D}^{\mathscr{X}}_n)$ with $z_{j}=\left(z_{1j},\dots,z_{nj} \right)\in\mathcal{D}^{\mathscr{X}}_n $ and $\sum_{j=1}^{m}\left\| a_j\right\| \leq1$. We arrive at
\begin{align*}
&\left\|\sum_{i=1}^n\left\langle x_i,\sum_{j=1}^{m}z_{ij}a_j\right\rangle \left\langle \sum_{j=1}^{m}z_{ij}a_j,x_i\right\rangle\right\|\\&\qquad=\left\|\sum_{j=1}^{m}\sum_{i=1}^n\left\langle x_i,z_{ij}\right\rangle a_j\left\langle \sum_{j=1}^{m}z_{ij}a_j,x_i\right\rangle\right\|\\&\qquad=	\left\|x\right\|^*_n\left\|\sum_{j=1}^{m}\left\|a_j\right\|\left( \sum_{i=1}^n\left\langle x_i,z_{ij}\right\rangle \frac{a_j\left\langle \sum_{j=1}^{m}z_{ij}a_j,x_i\right\rangle}{\left\|a_j\right\|\left\|x\right\|^*_n}\right) \right\|.
\end{align*}
 Set $b_{ij}=\frac{a_j\left\langle \sum_{j=1}^{m}z_{ij}a_j,x_i\right\rangle}{\left\|a_j\right\|\left\|x\right\|^*_n}$. One can easily verify that $\left(b_{1j},\dots,b_{nj}\right)\in\mathcal{L}^2_n\left(\mathfrak{A}\right)_{[1]}$. It follows from Lemma \ref{multinorm} that
 \begin{align*}
 	\left\|\sum_{i=1}^n\left\langle x_i,\sum_{j=1}^{m}z_{ij}a_j\right\rangle \left\langle \sum_{j=1}^{m}z_{ij}a_j,x_i\right\rangle\right\|\leq\left\|x\right\|^*_n\left\|x\right\|^{\mathscr{X}}_n.
 \end{align*}
 It is deduced from \eqref{dence} that $\left\|x\right\|^*_n\leq\left\|x\right\|^{\mathscr{X}}_n$. We conclude from Proposition \ref{module}(1) that $\left\|x\right\|^*_n=\left\|x\right\|^{\mathscr{X}}_n$.
\end{proof}

\section{Application to the decomposition of a Hilbert $C^*$-module}
Dales and Polyakov \cite{2012} presented the notions of small and orthogonal decompositions with respect to a multi-normed space $\left(\left( \mathcal{X}^n,\left\|\cdot\right\|_n\right):n\in\mathbb{N}\right)$. We study these notions in the setting of Hilbert $C^*$-modules. Suppose that $\mathscr{X}$ is a Hilbert $\mathfrak{A}$-module. We recall $\mathscr{X} = \mathscr{X}_1\oplus \cdots \oplus \mathscr{X}_n$ a direct sum decomposition
of $\mathscr{X}$ if $\mathscr{X}_1, \dots, \mathscr{X}_n$ are complemented (closed) submodules of $\mathscr{X}$ by allowing the possibility that $\mathscr{X}_i=0$
for some $1\leq i\leq n$. The decomposition $\mathscr{X}=\mathscr{X}_1\oplus\dots\oplus \mathscr{X}_n$ is hermitian if $\left\|\alpha_1x_1+\dots+\alpha_nx_n\right\|=\left\|x_1+\dots+x_n\right\|$, where $\alpha_i\in\mathbb{C}$ with $\left| \alpha_i\right|=1 $ and $x_i\in \mathscr{X}_i\,\;(1\leq i\leq n)$.
Let $\mathfrak{A}$ be a unital $C^*$-algebra. Set $\mathscr{X}=\mathit{l}_2^2({\mathfrak{A}})$, $\mathscr{X}_1=\{(a,0):a\in\mathfrak{A}\}$, and $\mathscr{X}_2=\{(0,b):b\in\mathfrak{A}\}$. Clearly $\mathscr{X}_1$ and $\mathscr{X}_2$ are complemented submodules of $\mathscr{X}$. The direct sum decomposition 
$\mathscr{X}=\mathscr{X}_1\oplus \mathscr{X}_2$ is hermitian.
Moreover, we can consider $\mathscr{Y}_1=\{(a,a):a\in\mathfrak{A}\}$ and $\mathscr{Y}_2=\{(0,b):b\in\mathfrak{A}\}$. Clearly $\mathscr{Y}_1$ and $\mathscr{Y}_2$ are complemented submodules of $\mathscr{X}$. Assume that the direct sum decomposition 
$\mathscr{X}=\mathscr{Y}_1\oplus \mathscr{Y}_2$ is hermitian. Then $1=\left\|\left( 1_{\mathfrak{A}},1_{\mathfrak{A}}\right)- \left( 0,1_{\mathfrak{A}}\right) \right\|=\left\|\left( 1_{\mathfrak{A}},1_{\mathfrak{A}}\right)+ \left( 0,1_{\mathfrak{A}}\right) \right\|=\sqrt{5}$, but this is a contradiction. We now consider decompositions related to multi-norms. Suppose that $\left( \left( {\mathscr{X}}^n,\left\|\cdot\right\|_n\right):n\in\mathbb{N} \right) $ is a multi-normed space. We consider $P_i : \mathscr{X}\to\mathscr{X}_i\,\;(1\leq i\leq n)$ for the natural projections \cite[p. 21]{lance}. A direct sum decomposition $\mathscr{X}=\mathscr{X}_1\oplus\dots\oplus \mathscr{X}_n$ is \emph{small} if for all $x_1,\dots,x_n\in \mathscr{X}$ we have $\left\|P_1x_1+\dots+P_nx_n\right\|\leq \left\|\left(x_1,\dots,x_n\right)\right\|_n$, and the decomposition is \emph{orthogonal} if 
$\left\| (y_1,\dots, y_j)\right\|_j = \left\| (x_1,\dots, x_n)\right\|_n$, where $x_i\in \mathscr{X}_i\,\;(1\leq i\leq n)$ and for all partition $\{S_j:1\leq j\leq m\}$ of $\{1,\dots,n\}$ such that $y_j =\sum_{i\in S_j}x_i\,\;(1\leq j\leq m)$. Let $\mathscr{X}=\mathscr{X}_1\oplus\dots\oplus \mathscr{X}_n$ be an orthogonal decomposition of $\mathscr{X}$. Then in particular, $\left\| (x_1,\dots, x_n)\right\|_n=\left\|x_1+\dots+ x_n \right\|$. Similar to \cite[Theorem 7.20]{2012}, one can show that a small decomposition is always orthogonal and an orthogonal decomposition is hermitian. 

Now, we extend \cite[Theorems 7.29 and 7.30]{2012}, which were demonstrated for Hilbert spaces and the space $C(\mathbb{K})$ of all complex-valued,
continuous functions on the compact space $\mathbb{K}$.
\begin{theorem}\label{decomposition}
Suppose that $\mathscr{X}$ is a Hilbert $C^*$-module. Let $\mathscr{X} = \mathscr{X}_1\oplus\dots\oplus\mathscr{X}_n$ be a direct sum decomposition
of $\mathscr{X}$. Then the following statements are equivalent:
\begin{enumerate}
\item\label{1}
	For $1\leq i\neq j\leq n$, we have $\mathscr{X}_i\perp	\mathscr{X}_j$;
\item\label{2}
The decomposition is small with respect to $\left( \left\|\cdot\right\|_n^{\mathscr{X}}:n\in\mathbb{N}\right)$;
\item\label{3}
The decomposition is small with respect to $\left( \left\|\cdot\right\|_n^{*}:n\in\mathbb{N}\right)$;
\item\label{4}
The decomposition is small with respect to the $\left( \left\|\cdot\right\|_n^{\mathcal{P}\left(\mathfrak{A} \right) }:n\in\mathbb{N}\right)$;
\item\label{5}
The decomposition is orthogonal with respect to $\left( \left\|\cdot\right\|_n^{\mathscr{X}}:n\in\mathbb{N}\right)$;
\item\label{6}
The decomposition is orthogonal with respect to $\left( \left\|\cdot\right\|_n^{*}:n\in\mathbb{N}\right)$;
\item\label{7}
The decomposition is orthogonal with respect to $\left( \left\|\cdot\right\|_n^{\mathcal{P}\left(\mathfrak{A} \right) }:n\in\mathbb{N}\right)$;
\item\label{8}
The decomposition is hermitian.
\end{enumerate}
\end{theorem}
\begin{proof}
By using the technique applied in \cite[Theorem 7.20]{2012} and \eqref{Hilbert multi norm}, we have (\ref{1}) $\Rightarrow$ (\ref{2}) $\Rightarrow$ (\ref{5}) $\Rightarrow$ (\ref{8}).
	 
Suppose that $\mathscr{X} = \mathscr{X}_1\oplus\dots\oplus\mathscr{X}_n$ is a hermitian decomposition. Pick $x\in\mathscr{X}_i$ and $y\in\mathscr{X}_j$, where $1\leq i,j\leq n$. Since the decomposition is hermitian, we observe that $\left\|x-y \right\|= \left\|x+y \right\|$. it follows that $\left\langle x,y\right\rangle=-\left\langle y,x\right\rangle$. Similarly, we arrive at $\left\langle ix,y\right\rangle=-\left\langle y,ix\right\rangle$. Thus, $\left\langle x,y\right\rangle=0$ and condition (\ref{1}) is established.

 Let $\mathscr{X} = \mathscr{X}_1\oplus\dots\oplus\mathscr{X}_n$, where $\mathscr{X}_i$ are orthogonal complemented submodules. Theorem \ref{decomposition} and Proposition \ref{module} ensure that the decomposition is small with respect to $\left( \left\|\cdot\right\|^*_n:n\in\mathbb{N}\right)$ and $\left( \left\|\cdot\right\|^{\mathcal{P}\left(\mathfrak{A}\right)}_n:n\in\mathbb{N}\right)$. Therefore, (\ref{1}) $\Rightarrow$ (\ref{3}) $\Rightarrow$ (\ref{6}) $\Rightarrow$ (\ref{8}) $\Rightarrow$ (\ref{1}) and (\ref{1}) $\Rightarrow$ (\ref{4}) $\Rightarrow$ (\ref{7}) $\Rightarrow$ (\ref{8}) $\Rightarrow$ (\ref{1}). 
\end{proof}
The following corollaries are consequences of Theorem \ref{decomposition}.

\begin{corollary}\cite[Theorem 7.29]{2012}
	Let $\mathbb{K}$ be a compact space, and let $\mathscr{X}=C(\mathbb{K})$. It follows that $\mathscr{X}_i=C(\mathbb{K}_i) \,\;(1\leq i\leq n)$ for some partition $\left\lbrace \mathbb{K}_1,\dots,\mathbb{K}_n\right\rbrace $ of clopen subspaces of $\mathbb{K}$.
\end{corollary}

\begin{corollary}\cite[Theorem 7.30]{2012}
Suppose that $\mathscr{H}$ is a Hilbert space and that $\mathscr{H} = \mathscr{H}_1\oplus\dots\oplus\mathscr{H}_n$ is a direct sum decomposition
	of $\mathscr{H}$. Then the following statements are equivalent:
	\begin{enumerate}
		\item
	For $1\leq i\neq j\leq n$, we have $\mathscr{H}_i\perp	\mathscr{H}_j$;
		\item
		The decomposition is small with respect to the Hilbert multi-norm;
		\item
		The decomposition is orthogonal with respect to the Hilbert multi-norm;
		\item
		The decomposition is hermitian.
	\end{enumerate}
\end{corollary}

\medskip
\noindent \textit{Conflict of Interest Statement.} On behalf of all authors, the corresponding author states that there is no conflict of interest.\\

\noindent\textit{Data Availability Statement.} Data sharing not applicable to this article as no datasets were generated or analysed during the current study.
\medskip
\bibliographystyle{amsplain}

\end{document}